\documentclass[letterpaper, 10 pt, conference]{ieeeconf}  

\IEEEoverridecommandlockouts                              
\usepackage{graphicx} 
\usepackage{color}
\usepackage{amsmath} 
\usepackage{amssymb}  
\usepackage{array}
\usepackage{multirow}

\newtheorem{thm}{Theorem}
\newtheorem{proposition}[thm]{Proposition}
\newtheorem{lemma}[thm]{Lemma}
\newtheorem{definition}{Definition}
\newtheorem{assumption}{Assumption}

\title{On the Ball-Marsden-Slemrod obstruction for bilinear control systems}

\author{
\authorblockN{Nabile Boussa\"{i}d}
\authorblockA{Laboratoire de Math\'ematiques de Besan\c{c}on, UMR 6623 \\
 Universit\'e de Bourgogne Franche-Comt\'e, Besan\c{c}on, France\\
{\tt\small Nabile.Boussaid@univ-fcomte.fr}}
\authorblockN{Marco Caponigro}
\authorblockA{\'Equipe M2N\\
 Conservatoire National des Arts et M\'etiers, Paris, France\\
{\tt\small Marco.Caponigro@cnam.fr}}
\authorblockN{Thomas Chambrion}
\authorblockA{
Université de Lorraine, CNRS, Inria, IECL, Nancy, France\\
{\tt\small Thomas.Chambrion@univ-lorraine.fr}}
}

\begin{document}

\maketitle
\thispagestyle{empty}
\pagestyle{empty}

\begin{abstract}
In this paper we present an extension to the case of $L^1$-controls of a famous result by Ball--Marsden--Slemrod on the obstruction to the controllability of bilinear control systems in infinite dimensional spaces. 
\end{abstract}

 \section{INTRODUCTION}

\subsection{Bilinear control systems}

Let $X$ be a Banach space, $A:D(A)\to X$ a linear operator in $X$ with domain $D(A)$, $B:X\to X$ a linear bounded operator and $\psi_0$ an element in $X$. 

We consider a following bilinear control system on $X$
\begin{equation}\label{EQ_main}
\begin{cases}
 \dot \psi(t) &= A\psi(t) + u(t) B \psi(t),\\
\psi(0)&= \psi_0,
\end{cases}
\end{equation}
where  $u:[0,+\infty)\to \mathbf{R}$ is a scalar function representing the control.
\begin{assumption}\label{ASS_base}
The pair $(A,B)$ of linear operators in $X$ satisfies
\begin{enumerate}
\item[$1)$] the operator $A$ generates a $C^0$-semigroup of linear bounded operators on $X$.  \label{ASS_C0_semi_group}
\item[$2)$] the operator $B$ is bounded. \label{ASS_B_bounded}
\end{enumerate}
\end{assumption}

\begin{definition}
Let $(A,B)$ satisfy Assumption \ref{ASS_base} and let $T>0$. A function $\psi:[0,T]\to X$ is a \emph{mild solution} of (\ref{EQ_main}) if for every $t$ in $[0,T]$, 
\begin{equation}\label{EQ_Duhamel}
\psi(t)=e^{tA}\psi_0 +\int_{0}^t e^{(t-s)A} B \psi(s) u(s) \mathrm{d}s
\end{equation}
\end{definition}

Equation (\ref{EQ_Duhamel}) is often called Duhamel formula.  Existence and uniqueness {of mild solutions} for equation~\eqref{EQ_main} is given by the following result (see, for instance, Proposition 2.1 and Remark 2.7 in \cite{bms}).
\begin{proposition}
Assume that $(A,B)$ satisfies Assumption \ref{ASS_base}. Then, for every $
\psi_0$ in $X$, for every $u$ in $L^1_{loc}([0,+\infty), \mathbf{R})$, there 
exists a unique mild solution $t\mapsto \Upsilon^u_{t,0} \psi_0$ to the 
Cauchy problem (\ref{EQ_main}). 
Moreover, for every $\psi_0$ in $X$, the end-point mapping $\Upsilon_{\cdot,0}\psi_0:[0,+\infty) \times L^1_{loc}([0,+\infty),\mathbf{R}) \to X$ is continuous. 
\end{proposition}

\begin{definition}
Assume that $(A,B)$ satisfies Assumption \ref{ASS_base} and let $\mathcal{U}$ be a subset of $L^1_{loc}([0,+\infty),\mathbf{R})$. For every $\psi_0$ in $X$, the \emph{attainable set from $\psi_0$ with controls in $\mathcal{U}$} is defined as
$$
\mathcal{A}(\psi_0,\mathcal{U})=\bigcup_{T \geq 0} \bigcup_{u \in \mathcal{U}} \{ \Upsilon^u_{T,0} \psi_0 \}.
$$
\end{definition}

Our main result is the following property of the attainable set of system (\ref{EQ_main}) with $L^1$ controls.
\begin{thm}\label{THM_main_result}
Assume that $(A,B)$ satisfies Assumption~\ref{ASS_base}. Then, for every $\psi_0$ in $X$, the attainable set 
$\mathcal{A}(\psi_0,L^1_{loc}([0,+\infty),\mathbf{R}))$
from $\psi_0$ with $L^1_{loc}$ controls is contained in a countable union of compacts subsets of $X$.
\end{thm}

\subsection{The Ball--Marsden--Slemrod obstruction}
Our main result, Theorem \ref{THM_main_result} is an extension of the well-known Ball--Marsden--Slemrod obstruction to controllability (see also \cite{Illner}) which is as follows.  
\begin{thm}[Theorem 3.6 in \cite{bms}]\label{THM_BMS}
Assume that $(A,B)$ satisfies Assumption~\ref{ASS_base}. Then, for every $\psi_0$ in $X$, the attainable set $\mathcal{A}(\psi_0,\cup_{r>1}L^r_{loc}([0,+\infty),\mathbf{R}))$
from $\psi_0$ with $L^r_{loc}$ controls, $r>1$, is contained in a countable union of compacts subsets of $X$.
\end{thm}

A consequence of Theorem~\ref{THM_BMS} to the framework of the conservative bilinear Schr\"odinger equation is given by Turinici.
\begin{thm}[Theorem~1 in \cite{turinici}]\label{THM_Turinici}
Assume that $(A,B)$ satisfies Assumption \ref{ASS_base}. Then, for every $\psi_0$ in $X$, the set $\cup_{\alpha >0} \alpha\mathcal{A}(\psi_0,\cup_{r>1}L^r_{loc}([0,+\infty),\mathbf{R}))$ is contained in a countable union of compacts subsets of $X$.
\end{thm}

Theorems \ref{THM_main_result} and \ref{THM_BMS} are basically empty in the case in which $X$ is finite dimensional, { since,} in this case, $X$ itself is a countable union of compact sets. On the other hand,  when $X$ is infinite dimensional,  these results represent a strong topological obstruction to the exact controllability. Indeed, compact subsets of an infinite dimensional Banach space have empty interiors and so is a countable union of closed subsets with empty interiors (as a consequence of Baire Theorem).

Whether the non-controllability result of Ball, Marsden, and Slemrod, 
Theorem~\ref{THM_BMS} holds for $L^1$ control has been an open question for 
decades. Indeed, the  proof of Theorem~\ref{THM_BMS} does not apply 
directly  to the $L^1$ case. To see what fails let us briefly recall  
the method used in~\cite{bms} for the proof of Theorem~\ref{THM_BMS}. The first 
step is to write
\begin{align*}
\lefteqn{\mathcal{A}(\psi_0,\cup_{r>1}L^r_{loc}([0,+\infty),\mathbf{R}))}\\
&~~~~=\bigcup_{T \geq 0} \bigcup_{r>1} \bigcup_{u \in L^r([0,T],\mathbf{R})} \{ \Upsilon^u_{T,0} \psi_0 \} \\
&~~~~= \bigcup_{l \in \mathbf{N}} \bigcup_{m \in \mathbf{N}} \bigcup_{k \in \mathbf{N}} \bigcup_{0\leq t \leq l } \left (\bigcup_{\|u\|_{L^{1+1/m}}\leq k}
\{ \Upsilon^u_{t,0} \psi_0 \} \right).
\end{align*}
Hence it is sufficient to prove that, for every $(l,m,k)$ in $\mathbf{N}^3$, the set
$$
\mathcal{A}^{l,m,k}= \bigcup_{0\leq t \leq l } \left (\bigcup_{\|u\|_{L^{1+1/m}}\leq k}
\{ \Upsilon^u_{t,0} \psi_0 \} \right )
$$
has compact closure in $X$. To this end, one considers a sequence $(\psi_n)_{n\in \mathbf{N}}$  in $\mathcal{A}^{l,m,k}$, associated with a sequence of times $(t_n)_{n\in \mathbf{N}}$ in $[0,l]$ and a sequence of controls $(u_n)_{n \in \mathbf{N}}$ in the ball of radius $k$ of $L^{1+1/m}([0,+\infty), \mathbf{R})$. By compactness of $[0,l]$, up to extraction, one can assume that $(t_n)_{n\in \mathbf{N}}$ tends to $t_\infty$ in $[0,l]$. By Banach--Alaoglu--Bourbaki Theorem, the balls of $L^{1+1/m}([0,+\infty), \mathbf{R})$ are weakly (sequentially)
compact and, hence, up to extraction, one can assume that $(u_n)_{n\in \mathbf{N}}$ converges weakly in $L^{1+1/m}([0,+\infty), \mathbf{R})$ to some $u_\infty$. 
The hard step of the proof (Lemma 3.7 in \cite{bms}) is then to show that $\Upsilon^{u_n}_{t_n,0}\psi_0$ tends to $\Upsilon^{u_\infty}_{t_\infty,0}{\psi_0}$ as $n$ tends to infinity. 

A crucial point in the proof of Theorem~\ref{THM_BMS} given in~\cite{bms} is 
the fact that the closed balls of $L^p$, $p>1$ are weakly sequentially 
compact.  This is no longer true for the balls of $L^1$, and this prevents a 
direct extension of the proof of Theorem \ref{THM_BMS} to
Theorem~\ref{THM_main_result}.  Here we present a brief and self-contained proof of Theorem~\ref{THM_main_result} mainly based on Dyson expansions and basics compacteness properties on Banach spaces. 

An alternative proof of Theorem~\ref{THM_BMS},  not rely{ing} on the reflectiveness 
of the set of admissible controls,  has been recently given in~\cite{UP}. The proof 
applies to  a very large class  of controls (namely, Radon measures) which 
contains locally integrable  functions, and for its generality it is technically 
quite involved,
in contrast with the simplicity of the underlying ideas. It applies also, with 
minor modifications, to nonlinear problems~\cite{chambrion:hal-01876952}. 


\subsection{Content}
In this note we present a simple proof of Theorem 
\ref{THM_main_result}. However, historical reasons have made different 
communities  use incompatible terminologies and, in order to avoid ambiguities, 
we present in Section~\ref{SEC_basic_facts} a quick reminder of basic facts in 
Banach topologies. Section~\ref{SEC_Dyson} gives a short introduction to the 
classical Dyson expansion (Section~\ref{SEC_Dyson_definition}) and the proof of 
an instrumental compactness property (Section~\ref{SEC_Dyson_compact}). We 
conclude in Section~\ref{SEC_proof} with the proof of Theorem~
\ref{THM_main_result}.  

 \section{BASIC FACTS ABOUT TOPOLOGY IN BANACH SPACES}\label{SEC_basic_facts}
\subsection{Notations} \label{SEC_notation}

The Banach space $X$ is endowed with norm $\| \cdot \|$. For every $\psi_c$ in $X$ and every $r>0$, $B_X(\psi_c,r)$ denotes the ball of center $\psi_c$ and of radius $r$:
$$
B_X(\psi_c,r)=\{ \psi \in X | \| \psi-\psi_c\| < r\}. 
$$

In the following, all we need to know about generators of $C^0$-semigroup is the classical result  stated in Proposition \ref{PRO_Hille_Yosida} (see  Chapter VII of \cite{Hille_Phillips}).
\begin{proposition}\label{PRO_Hille_Yosida}
Assume that $A$ generates a $C^0$-semigroup. Then
 there exist $M,\omega>0$ such that $\|e^{At}\|\leq M e^{\omega t}$ for every $t\geq 0$.
\end{proposition}

\subsection{Compact subset of Banach spaces}

\begin{definition}
Let $X$ be a Banach space and $Y$ be a subset of $X$. A family $(O_i)_{i\in I}$ is an open cover of $Y$ if $O_i$ is open in $X$ for every $i$ in $I$ and $Y\subset \cup_{i \in I} O_i$.  
\end{definition}

\begin{definition}
Let $X$ be a Banach space. A subset $Y$ of $X$ is said to be \emph{compact} if
from any open cover of $Y$, it is possible to extract a finite cover of $Y$. 
\end{definition}

\begin{definition}
Let $X$ be a Banach space. A subset $Y$ of $X$ is said to be \emph{sequentially compact} if
from any sequence $(\psi_n)_{n\in \mathbf{N}}$ taking value in $Y$ , it is possible to extract a subsequence $(\psi_{\phi(n)})_{n \in \mathbf{N}}$ converging in $Y$. 
\end{definition}

\begin{definition}
Let $X$ be a Banach space. A subset $Y$ of $X$ is said to be \emph{totally 
bounded} if for every $\varepsilon>0$, there exist $N \in \mathbf{N}$ and a finite family $(x_i)_{1\leq i \leq N}$ in $X$  such that 
$$Y \subset \bigcup_{i=1}^N B_X(x_i, \varepsilon).$$ 
\end{definition}

\begin{proposition}\label{PRO_equivalence_compacites_topologique_sequentielle}
Let $X$ be a Banach space. For every subset $Y$ of $X$, the following assertions are equivalent:
\begin{enumerate}
\item $Y$ is compact.
\item $Y$ is sequentially compact.
\item $Y$ is complete and totally bounded.
\item $Y$ is closed and totally bounded.
\end{enumerate}   
\end{proposition}

\begin{proposition}\label{PRO_sum_compact}
Let $X$ be a Banach space, $N$ in $\mathbf{N}$ and $(Y_i)_{1\leq i \leq N}$ a finite family of compact subsets of $X$.  Then, the finite sum
$$
\sum_{i=1}^N Y_i =\{y_1 + y_2 +\ldots +y_N \mid y_i\in Y_i,  i = 1,\ldots,N\}
$$
is compact as well.
\end{proposition}

\begin{proposition}\label{PRO_sum_totally_bounded}
Let $X$ be a Banach space, $N$ in $\mathbf{N}$ and $(Y_i)_{1\leq i \leq N}$ a finite family of totally bounded subsets of $X$.  Then, the finite sum
$$
\sum_{i=1}^N Y_i =\{y_1 + y_2 +\ldots +y_N\mid y_i\in Y_i,   i = 1,\ldots,N\}
$$
is totally bounded as well.
\end{proposition}
%

\begin{proposition}\label{PRO_uniforme_continuite_eAB}
Let $X$ be a Banach space, $T>0$ and $(A,B)$ satisfies Assumption \ref{ASS_base}. Define the mapping
$$
\begin{array}{lccc}
F:&[0,T]\times [0,T]\times X & \to & X\\
 &(s,t,\psi)&  \mapsto &  e^{|t-s|A} B \psi 
\end{array}
$$
Then, for every totally bounded subset $Y$ of $X$, the set $F([0,T]\times [0,T]\times Y)$ is totally bounded as well.
\end{proposition} 
\begin{proof}
We claim that $G:(t,\psi)\mapsto e^{tA}\psi$ is  jointly continuous in its two variables. Indeed, for every $\psi,\psi_0$ in $X$, for every $t,t_0\geq 0$,
\begin{eqnarray*}
{\|e^{tA}\psi-e^{t_0 A} \psi_0 \| }
& \! \!\leq \! \!& \|e^{tA} (\psi-\psi_0) \| + \| (e^{tA}-e^{t_0A})\psi_0\|\\
&\! \! \leq \! \! & M e^{\omega t} \|\psi -\psi_0\| + \| (e^{tA}-e^{t_0A})\psi_0\|.
\end{eqnarray*} 
This last quantity tends to zero as $(t,\psi)$ tends to $(t_0,\psi_0)$.
As a consequence, $F$ is continuous (as composition of continuous functions). 

If $Y$ is totally bounded, the topological closure $\bar{Y}$ of $Y$ is compact (because the ambient space $X$ is complete). Hence $[0,T]\times [0,T] \times \bar{Y}$ is compact. By continuity, $F([0,T]\times [0,T] \times \bar{Y})$ is compact, hence is totally bounded. The set $F([0,T]\times [0,T]\times Y)$, which is contained in $F([0,T]\times [0,T] \times \bar{Y})$,  is, therefore, totally bounded as well.
\end{proof}

\subsection{Partition of unity in Banach spaces}

\begin{definition}
Let $X$ a Banach space. A family $(x_i)_{i \in I}$ of points of $X$ is \emph{locally finite} if for every $x$ in $X$ and every $R>0$, the cardinality of the set 
$$
\left ( \bigcup_{i \in I} \{x_i\} \right ) \cap B_X(x,R)
$$
is finite.
\end{definition}

\begin{definition}
Let $X$ be a Banach space, $Y$ be a subset of $X$, and $(O_i)_{i\in I}$ be an open cover of $Y$.
A family $(\phi_i)_{i \in I}$ of continuous functions from $Y$ to $[0,1]$ is called a \emph{partition of the unity of $Y$ adapted to the cover $(O_i)_{i\in I}$} 
if
\begin{itemize}
 \item[$(i)$]  for every $i \in I$, $\phi_i(x)=0$ for every $x \notin O_i$;
 \item[$(ii)$] $\sum_{i\in I}\phi_i(x)=1$ for every $x \in Y$.
\end{itemize}
\end{definition}

\begin{proposition}\label{PRO_partition_unity}
Let $X$ be a Banach space, $Y$ a subset of $X$, $\delta>0$, $(x_j)_{j\in J}$ a locally finite family of points in $Y$ such that $Y \subset \cup_{j \in J}B_X(x_j, \delta)$. Then, there exists $(\phi_j)_{j \in J}$ a partition of the unity adapted to 
the open cover $(B(x_j,2\delta))_{j \in J}$ of $Y$.

 Moreover, if a family $(\phi_j)_{j\in J}$ is a partition of the unity adapted to 
the open cover $(B_X(x_j,2\delta))_{j \in J}$, then  for every $x$ in $Y$, $\|x-\sum_{j\in J}\phi_j(x) x_j \| \leq 2\delta$.
\end{proposition}

\begin{proof}
We first prove the existence of a partition of the unity adapted to 
the open covering $(B_X(x_j,2\delta))_{j \in J}$ of $Y$.
To this end, we define, for every $j$ in $J$, the continuous functions
$\varphi_j:X\to [0,1]$ by 
$$
\begin{cases}
 \varphi_j(x)=1, &\mbox{ if } \|x-x_j\| < \delta,\\
 \phi_j(x)=2-\|x-x_j\|/\delta,  &\mbox{ if } \delta \leq \|x-x_j\| < 2\delta,\\
 \varphi_j(x)=0, &\mbox{ if }  2\delta \leq\|x-x_j\|.
\end{cases}
$$
Since  the family $(x_j)_{j\in J}$ is locally finite, the sum 
$\sum_{j\in J} \varphi_j(x)$ converges for every $x$ in $Y$. Moreover, since
$Y\subset \cup_{j\in J}B(x_j,\delta)$, the
function $x\mapsto \sum_{j\in J} \varphi_j(x)$   does not vanish on $Y$.
For every $j_0$ in $J$, we define
$\phi_{j_0}$ by
$$
\phi_{j_0}(x)= \varphi_{j_0}(x) \frac{1}{\sum_{j\in J} \varphi_j(x)},
$$
and the family $(\phi_j)_{j \in J}$ is a partition of the unity adapted to 
the open cover $(B_X(x_j,2\delta))_{j \in J}$ of $Y$. 

We now prove the second point of Proposition \ref{PRO_partition_unity}.
Let $(\phi_j)_{j\in J}$ be a partition of unity of $Y$ adapted to the cover  $(B_X(x_j,2\delta))_{j \in J}$. Then, for every 
$x$ in $Y$,
\begin{eqnarray*}
\left \|x -\sum_{j \in J} \phi_j(x) x_j \right \| & =& \left \|\sum_{j\in J} \phi_j(x) x -\sum_{j \in J} \phi_j(x) x_j \right \|\\
&=& \left \|\sum_{j\in J} \phi_j(x) (x - x_j) \right \|\\
&\leq & \sum_{j\in J} \phi_j(x) \left \| x - x_j \right \|. 
\end{eqnarray*}   
By construction, $\phi_j(x)=0$ as soon as $\| x - x_j \|  \geq 2\delta$. Hence,
\begin{eqnarray*}
\left \|x -\sum_{j \in J} \phi_j(x) x_j \right \| & \leq & 2\delta \sum_{j\in J} \phi_j(x)\leq  2\delta,
\end{eqnarray*}
which concludes the proof.
\end{proof}

 \section{DYSON EXPANSION} \label{SEC_Dyson}

\subsection{The Dyson Operators}\label{SEC_Dyson_definition}
For every $u$ in $L^1_{loc}([0,+\infty),\mathbf{R})$,  $p \in \mathbf{N}$, and $t\geq 0$ we define the linear bounded operator $W_p(t,u):X\to X $ recursively by
\begin{align*}
W_0(t,u)\psi & =e^{(t-s)A}\psi\\
W_p(t,u)\psi & =\int_0^t e^{(t-s)A}B W_{p-1}(s,u)\psi u(s) \mathrm{d}s, \quad \! \mbox{\,for\,} p\geq 1,
\end{align*}
for every $\psi$ in $X$.
We have the following estimate on the norm of the operator.
\begin{proposition}\label{PRO_borne_Wp}
For every $u$ in $L^1_{loc}([0,+\infty),\mathbf{R})$,  $p \in \mathbf{N}$, and $t\geq 0$ 
\[
\|W_p(t,u)\| \leq \frac{Me^{\omega t} \|B\|^{p} (\int_{0}^t|u(s)|\mathrm{d}s)^p}{p!}.
\]
\end{proposition}
\begin{proof}
We prove the result by induction on $p$ in $\mathbf{N}$.
For $p=0$ the result clearly follows from Proposition~\ref{PRO_Hille_Yosida}.
Assume that the result holds for $p\geq 0$.   Then, for every $\psi$ in $X$,
\begin{align*}
\lefteqn{\|W_{p+1}(t,u) \psi \| }\\
& ~~~~\leq  \int_0^t e^{(t-s)A}B W_{p}(s,u)\psi u(s) \mathrm{d}s\\
&~~~~\leq M \int_0^t e^{(t-s)\omega} \| B \| \frac{e^{\omega s} \|B\|^{p} (\int_{0}^s|u(\tau)|\mathrm{d}\tau)^p}{p!}u(s) \mathrm{d}s\\
&~~~~\leq M e^{\omega t} \|B \|^{p+1} \frac{ (\int_{0}^t|u(\tau)|\mathrm{d}\tau)^{p+1}}{(p+1)!}.
\end{align*}
The last inequality follows from Proposition \ref{PRO_Hille_Yosida}. We conclude the proof by induction on $p$.
\end{proof}

\subsection{A compactness property}\label{SEC_Dyson_compact}
\begin{lemma}\label{PRO_calWj_totally_bounded} 
For every $j$ in $\mathbf{N}$, $T\geq 0$ and $K\geq 0$, and $\psi$ in $X$ the set
 \begin{equation}\label{DEF_calW}
\mathcal{W}_j^{T,K}=\{W_j(t,u) \psi \mid 0\leq t \leq T, \|u\|_{L^1}\leq K\},
\end{equation}
is totally bounded 
\end{lemma}
\begin{proof}
We prove the result by induction on $j$ in $\mathbf{N}$.
For $j=0$, consider 
$\mathcal{W}_0^{T,K}=\{e^{tA}\psi, 0\leq t 
\leq T\}$ and let $(w_n)_{n\in \mathbf{N}}$ be a sequence in $\mathcal{W}_0^{T,K} 
$. Then there exists a sequence $(t_n)_{n\in \mathbf{N}}$ such that $w_n=e^{t_n A} 
\psi$ for every $n$. Up to extraction $\lim_{n\to 
\infty} t_n=t\in [0,T]$ since $[0,T]$ is compact. By definition of $C^0$-semigroup, $\lim_{n\to \infty} e^{t_n A}\psi =e^{tA}\psi$. This proves that 
$\mathcal{W}_0^{T,K}$ is sequentially compact, hence compact  and, in particular, totally 
bounded (Proposition \ref{PRO_equivalence_compacites_topologique_sequentielle}).

Assume that, for $j\geq 0$, $\mathcal{W}_j^{T,K}$ is totally bounded. 
By Proposition~\ref{PRO_uniforme_continuite_eAB}, the set
\begin{eqnarray*}
Z_j^{T,K}&:= &\{e^{(t-s)A} B \psi, \psi \in W_j^{T,K}, 0\leq s \leq t \leq T\}  \\
&\subset & F([0,T]^2 \times \mathcal{W}_j^{T,K})
\end{eqnarray*}
is totally bounded as well. 

Let $\varepsilon>0$ be given and define $\delta= \frac{\varepsilon}{2K+1}>0$. Since $Z_j^{T,K}$ is totally bounded, there exists a finite family $(x_i)_{1\leq i \leq N_{\delta}}$ in $Z_j^{T,K}$ such that
$$
Z_j^{T,K} \subset \bigcup_{i=1}^{N_\delta} B_X(x_i, \delta).
$$  
Let $(\phi_i)_{\leq i \leq N_\delta}$ be a partition of the unity adapted to the cover 
$\cup_{i=1}^{N_\delta} B(x_i, 2\delta)$ of $Z_j^{T,K} $. Such a partition of the unity exists by Proposition~\ref{PRO_partition_unity}, and moreover, for every $x$ in $Z_j^{T,K}$, we have 
\begin{equation}\label{EQ_major_part_unity}
\left \|x -\sum_{i=1}^{N_\delta} \phi_i(x) x_i \right \| \leq 2\delta.
\end{equation}
Applying the inequality (\ref{EQ_major_part_unity}) with $x= e^{(t-s)A} B W_{j}(s,u)\psi_0$, we get, for every $u$ in $L^1$ and every $(s,t)$ such that $0\leq s\leq t$,
\begin{eqnarray*}
\lefteqn{\left \|  e^{(t-s)A} B W_{j}(s,u)\psi_0  \right.}\\
&~~~~~~~~~& \left.  - \sum_{i=1}^{N_\delta}  \phi_i(e^{(t-s)A} B W_{j}(s,u)\psi_0 )  x_i \right \|\leq 2\delta.
\end{eqnarray*}
Multiplying by $u(s)$ and integrating for $s$ in $[0,t]$, one gets for $\|u\|_{L^1}\leq K$
 \begin{eqnarray*}
 \lefteqn{\left \|\int_0^t e^{(t-s)A} B W_{j}(s,u)\psi_0 u(s)\mathrm{d}s -\right.}\\
 &&\left. \! \sum_{i=1}^{N_\delta} \int_0^t \!\!\! \phi_i(e^{(t-s)A} B W_{j}(s,u)\psi_0 ) u(s) \mathrm{d}s \, x_i \right \| \leq 2\delta \|u\|_{L^1},
 \end{eqnarray*}
 that is
 \begin{eqnarray}\label{EQ_major_precompact_Z}
 \lefteqn{\left \|\int_0^t e^{(t-s)A} B W_{j}(s,u)\psi_0 u(s)\mathrm{d}s \right. }  \\
 & \!\!&\left . - \sum_{i=1}^{N_\delta} \int_0^t \!\!\phi_i(e^{(t-s)A} B W_{j}(s,u)\psi_0 ) u(s) \mathrm{d}s \, x_i \right \| \leq 2\delta K,\nonumber
 \end{eqnarray}
The set $\sum_{i=1}^{N_\delta} [0,K] x_i$ is compact by Proposition \ref{PRO_sum_compact} and, hence, totally bounded. 
Then there exists a finite family 
$(y_i)_{1 \leq i \leq  N_\delta'}$ such that 
\begin{equation}\label{EQ_maillage_pave}
\sum_{i=1}^{N_\delta} [0,K] x_i \subset \bigcup_{i=1}^{N_\delta'} B_X(y_i,\delta).
\end{equation}
From (\ref{EQ_major_precompact_Z}) and (\ref{EQ_maillage_pave}), one deduces that 
$$
\mathcal{W}_{j+1}^{T,K} \subset \bigcup_{i=1}^{N_\delta'} B_X(y_i, (2K+1)\delta)
= \bigcup_{i=1}^{N_\delta'} B_X(y_i, \varepsilon).
$$
This proves that $\mathcal{W}_{j+1}^{T,K}$ is totally bounded and concludes the proof.
\end{proof}

\subsection{Convergence of the Dyson expansion}

\begin{proposition}
For every $u$ in $L^1([0,+\infty), \mathbf{R})$, $t\geq 0$, and $\psi_0 \in X$ 
$$
\left \| \Upsilon^u_{t,0}\psi_0 \right \| \leq 
M e^{\omega t} \|\psi_0\|\exp \left (M e^{\omega t} \|B\| \int_0^t |u(s)|\mathrm{d}s \right )  
$$
\end{proposition}
\begin{proof}
The proof follows the proof of~\cite[Theorem 2.5]{bms}.
By Duhamel formula (\ref{EQ_Duhamel}) and Proposition \ref{PRO_Hille_Yosida},
\begin{align*}
\left \| \Upsilon^u_{t,0}\psi_0 \right \|&\\
=&\left \| e^{tA}\psi_0 +
\int_0^t e^{(t-s)A} B \Upsilon^u_{s,0} \psi_0 u(s) \mathrm{d}s
\right \|  e^{t\omega }\psi_0\\
\leq & M e^{t\omega }\| \psi_0 \| + \int_0^t M e^{\omega (t-s)} \|B\| 
\left  \| \Upsilon^u_{s,0}\psi_0 \right \| u(s) \mathrm{d} s,
\end{align*}
and the conclusion follows by Gronwall lemma (see~\cite[Lemma 2.6]{bms}).
\end{proof}

 \addtolength{\textheight}{-100pt}

\begin{proposition}\label{PRO_limite_zero}
For every $u$ in $L^1([0,+\infty), \mathbf{R})$, $p$ in $\mathbf{N}$,  $t\geq 0$, and $\psi_0$ in $X$
$$
{\lim_{p\to \infty} \left \|\int_0^te^{(t-s)A} B W_p(s,t)\Upsilon^u_{s,0}\psi_0 u(s)\mathrm{d}s \right \| =0}.
$$
\end{proposition}
\begin{proof}
Consider 
 \begin{eqnarray*}
 \lefteqn{\left \|\int_0^te^{(t-s)A} B W_p(s,t)\Upsilon^u_{s,0}\psi_0 u(s)\mathrm{d}s \right \|}\\
 &\leq & \int_0^t \|e^{(t-s)A}\| \|B\| \|W_p(s,t)| \Upsilon^u_{s,0}\psi_0\| |u(s)|\mathrm{d}s
 \end{eqnarray*}
 and recall that $\|W_p(s,t) \|$ tends to zero as $p$ tends to infinity (Proposition \ref{PRO_borne_Wp}).
\end{proof}

\begin{proposition}\label{PRO_Dyson_expansion}
For every $u$ in $L^1_{loc}([0,+\infty),\mathbf{R})$, $p$ in $\mathbf{N}$, $t\geq 0$, and $\psi_0$ in $X$
$$
\Upsilon^u_{t,0}\psi_0= \sum_{p=0}^\infty W_p(t,0,u)\psi_0.
$$
\end{proposition}
\begin{proof} Applying iteratively $p$-times Duhamel formula (\ref{EQ_Duhamel}), one gets
 \begin{eqnarray*}
 \Upsilon^u_{t,0}\psi_0&=& e^{tA}\psi_0 +\int_{s_1=0}^t \!\!\!\! e^{(t-s_1)A}B u(s_1) \Upsilon^u_{s_1,0} \psi_0 \mathrm{d} s_1\\
 &=& e^{tA}\psi_0 +\int_{s_1=0}^t \!\!\!\! e^{(t-s_1)A}B u(s_1) e^{s_1 A}\psi_0 \mathrm{d} s_1\\
 &&\lefteqn{+ \int_{s=0}^t \!\!\!\!e^{(t-s_)A}B W_1(s,t) \Upsilon^u_{s,0} \psi_0 u(s)u(s)\mathrm{d} s}\\
 &=&\sum_{j=1}^{p} W_p(t,0)\psi_0 \\
 &&+\int_{0}^t e^{(t-s)A} B W_{p}(s,t) \Upsilon^u_{s,0}\psi_0 u(s) \mathrm{d}s.
 \end{eqnarray*} 
Hence, for every $p\geq 1$, 
\begin{eqnarray*}
\lefteqn{\Upsilon^u_{t,0}\psi_0 - \sum_{j=0}^{p} W_j(t,0)\psi_0} \\
&~~~~&= \int_{0}^t e^{(t-s)A} B W_{p}(s,t) \Upsilon^u_{s,0}\psi_0 u(s) \mathrm{d}s
\end{eqnarray*}
and the result follow from Proposition \ref{PRO_limite_zero} as $p$ tends to $\infty$.
\end{proof}

 \section{PROOF OF THEOREM \ref{THM_main_result}}\label{SEC_proof}

We proceed now to the proof of Theorem \ref{THM_main_result}.
First of all, notice that, for every $\psi_0$ in $X$,
\begin{eqnarray*}
\lefteqn{\mathcal{A}(\psi_0,L^1_{loc}([0,+\infty),\mathbf{R})}\\
&\quad \quad \quad \quad=&\bigcup_{l\in \mathbf{N}} \bigcup_{m\in \mathbf{N}}  \{\Upsilon^u_{t,0}\psi_0, \|u\|_{L^1} \leq l, 0\leq t \leq m\},
\end{eqnarray*}
and it is enough to prove that, for every $l$ and $m$ in $\mathbf{N}$, the set
$$
\{\Upsilon^u_{t,0}\psi_0, \|u\|_{L^1} \leq l, 0\leq t \leq m\}
$$
is totally bounded.

Let $\varepsilon>0$.
From the convergence of the Dyson expansion (Proposition \ref{PRO_Dyson_expansion}) and the bound on the operators $W_j$ (Proposition \ref{PRO_borne_Wp}), there exists a  integer $N_\varepsilon$ such that
\begin{equation}\label{EQ_majoration_queue}
\left \|\sum_{p\geq N_\varepsilon} W_p(t,u)\psi_0 \right \| \leq \frac{\varepsilon}{2},
\end{equation}
for every $t$ in $[0,m]$ and every $u$ such that $\|u\|_{L^1}\leq l$. 
For each  $j=1,\ldots, N_\varepsilon$  the sets $\mathcal{W}_j^{m,l}$, defined by \eqref{DEF_calW}, are totally bounded (Lemma \ref{PRO_calWj_totally_bounded}), hence their sum
$$
\sum_{j=0}^{N_\varepsilon} \mathcal{W}_j^{m,l} 
$$
is totally bounded as well (Proposition \ref{PRO_sum_totally_bounded}). Hence 
there exists a family $(x_i)_{1\leq i \leq N_1}$ of points of $ \sum_{j=0}^{N_\varepsilon} \mathcal{W}_j^{m,l} $ such that
\begin{equation}\label{EQ_somme_cal Wj_precompact}
\sum_{j=0}^{N_\varepsilon} \mathcal{W}_j^{m,l}  \subset \bigcup_{i=1}^{N_1}
B_X\left (x_i,\frac{\varepsilon}{2} \right ).
\end{equation}
Gathering (\ref{EQ_majoration_queue}) and (\ref{EQ_somme_cal Wj_precompact}), one gets 
$$
\sum_{j=0}^{\infty} \mathcal{W}_j^{m,l} \subset \bigcup_{i=1}^{N_1}
B_X\left (x_i,{\varepsilon} \right ),  
$$
which concludes the proof of Theorem~\ref{THM_main_result}.

 \section{ACKNOWLEDGMENTS}

This work has been supported by the project DISQUO of the DEFI InFIniTI 2017 by
CNRS and the QUACO project by ANR 17-CE40-0007-01.

\bibliographystyle{alpha}
\bibliography{biblio}

\end{document}